\documentclass[11pt,oneside]{amsart}

\title[Conjugacy and homogeneous structures]{The conjugacy problem for automorphism groups of countable homogeneous structures}

\author{Samuel Coskey} \address{Samuel Coskey, Department of Mathematics, Boise State University, 1910 University Drive, Boise, ID, 83725}
\email{scoskey@nylogic.org}
\urladdr{boolesrings.org/scoskey}

\author{Paul Ellis} \address{Paul Ellis, Department of Mathematics and Computer Science, Manhattanville College, 2900 Purchase Street, Purchase, NY, 10577}
\email{paulellis@paulellis.org}
\urladdr{paullellis.org}

\usepackage[marginratio=1:1]{geometry}
\usepackage{setspace}
\usepackage{amssymb,mathpazo,bm}
\usepackage{tikz}

\newcommand{\NN}{\mathbb N}
\newcommand{\QQ}{\mathbb Q}
\newcommand{\ZZ}{\mathbb Z}
\newcommand{\RR}{\mathbb R}
\newcommand{\Aut}{\mathop{\mathrm{Aut}}}
\newcommand{\im}{\mathop{\mathrm{im}}}
\newcommand{\set}[1]{\left\{\,#1\,\right\}}
\newcommand{\sset}{\mathsf{set}}

\DeclareMathOperator{\Mod}{Mod}

\usepackage{etoolbox}
\makeatletter\pretocmd{\@seccntformat}{\S}{}{}\pretocmd{\@subseccntformat}{\S}{}{}\makeatother

\newtheorem{thm}{Theorem}[section]
\newtheorem{lem}[thm]{Lemma}

\begin{document}
\maketitle

\begin{abstract}
  We consider the conjugacy problem for the automorphism groups of a number of countable homogeneous structures. In each case we find the precise complexity of the conjugacy relation in the sense of Borel reducibility.
\end{abstract}

\section{Introduction}

In \cite{summer}, we showed together with Scott Schneider that the conjugacy problem for the automorphism group of the random graph is Borel complete. In this article we aim to continue this work and examine the complexity of the conjugacy problem for a variety of countable homogeneous structures. We begin by giving a brief overview of the above concepts.

Let $\mathcal L=\{R_i\}$ be a countable set of relation symbols, where each $R_i$ has arity $n_i$. Then the \emph{space of countable $\mathcal L$-structures} is given by
\[\Mod_{\mathcal L}=\prod\mathcal P(\NN^{n_i})\text{.}
\]
Here, $\Mod_{\mathcal L}$ has the product topology, and each factor has the natural Cantor set topology. Following Friedman--Stanley \cite{FS} and Hjorth--Kechris \cite{HK}, we identify the \emph{classification problem} for a set of $\mathcal L$-structures $C\subset\Mod_{\mathcal L}$ with the \emph{isomorphism equivalence relation} on $C$. In this article we will most often consider the language $\mathcal L=\{R\}$ where $R$ is a binary relation, and classes $C$ such as the countable undirected graphs, digraphs, linear orderings, and so on.

In order to weigh the relative complexity of such classification problems, we use the following notion of reducibility between equivalence relations. First, recall that a Borel structure on a set $X$ is said to be \emph{standard} if it arises as the Borel $\sigma$-algebra of a separable, completely metrizable topology on $X$. Now if $E,F$ are equivalence relations on standard Borel spaces $X,Y$, then $E$ is said to be \emph{Borel reducible} to $F$, written $E\leq_BF$, if there exists a Borel function $f\colon X\to Y$ such that for all $x,x'\in X$,
\[x\mathrel{E}x'\iff f(x)\mathrel{F}f(x')\text{.}
\]
Intuitively, if you have a set of complete invariants for $F$, and if $E\leq_BF$, then by composing with the reduction function $f$ you can use the same invariants for $E$ as well.

If $E$ is Borel reducible to the equality relation on some (any) standard Borel space, then $E$ is said to be \emph{smooth} or completely classifiable. On the other end of the spectrum, if $E$ has the property that for any countable language $\mathcal L$ and any Borel class $C\subset\Mod_{\mathcal L}$ the isomorphism relation on $C$ is reducible to $E$, then $E$ is said to be \emph{Borel complete}. We remark that if $E$ is a Borel complete equivalence relation then $E$ is necessarily a non-Borel subset of $X\times X$ \cite{FS}.

We will use the following examples of Borel complete equivalence relations. The result is essentially folklore.

\begin{thm}
  The isomorphism equivalence relation on each of the following classes of countable structures is Borel complete:
\begin{itemize}
\item linear orders;
\item tournaments;
\item $K_n$-free graphs, where $K_n$ is the complete graph on $n$ vertices and $n\geq 3$; and
\item $\mathcal{F}$-avoiding digraphs, where $\mathcal{F}$ is a family of finite tournaments, each of size $\geq 3$.
\end{itemize}
\end{thm}

\begin{proof}
  The isomorphism relation on countable linear orders is Borel complete by Theorem~3 of \cite{FS}. Since any linear order is in particular a tournament, it follows that the isomorphism relation on countable tournaments is Borel complete too. For a nice presentation of a proof that the isomorphism relation on countable graphs is Borel complete, see Theorem~13.1.2 of \cite{gao}. The ``tag'' used in this proof can be easily modified to show Borel completeness for the isomorphism relation on the remaining two classes.
\end{proof}

In this article we will also study the \emph{conjugacy problem}, or the problem of deciding whether two elements in a given group are conjugate. As before, we identify the conjugacy problem for $G$ with the conjugacy equivalence relation on $G$. When $G$ is the automorphism group of a countable $\mathcal{L}$-structure $M$, this equivalence relation is actually a special case of the isomorphism equivalence relations described above. Indeed, we can identify $\Aut(M)$ with the class $C\subset\Mod_{\mathcal L\cup\{R\}}$ of all expansions $(M;R^f)$ where $R^f$ is the binary relation which is the graph of the automorphism $f$. Then two elements of $\Aut(M)$ are conjugate if and only if the corresponding structures in $C$ are isomorphic.

We will study the conjugacy problem only for structures that are \emph{homogeneous}. A structure is homogeneous if every finite partial automorphism can be extended to a full automorphism. We direct the reader's attention to the survey \cite{macpherson} for a good overview of countable homogeneous structures. We will give several examples of homogeneous structures at the beginning of each subsequent section.

Homogeneous structures and their automorphisms have been studied a great deal from the point of view of model theory and algebra; for a survey of a portion of this work see \cite{lascar}. More recently, a deep connection between structural Ramsey theory and the topological dynamics of such groups has been explored, as detailed in \cite{kpt} and numerous subsequent articles.

Returning to conjugacy, after the results summarized in \cite{summer} we formulated a conjecture that the conjugacy problem for automorphism groups of countable homogeneous structures is always either smooth (for ``trivial'' homogeneous structures like $\NN$ with no relations) or Borel complete (for ``complicated'' homogeneous structures like the random graph). After studying further examples, we observe that this pattern \emph{mostly} holds, even though we found an exception in Theorem~\ref{thm:Ginfty}. It is our hope that a model theorist will look upon our results with a knowing wink and suggest or prove the right conjecture.

In Section~\ref{sec:LinearAndLocalOrders}, we sketch the proof that the conjugacy problem for countable homogeneous linear orders is Borel complete.  We also introduce local orders (and, more generally, the structures $S(n)$) and solve the analogous problem for them.  In Section~\ref{sec:UndirectedGraphs}, we treat countable homogeneous simple undirected graphs. In Section~\ref{sec:Digraphs} we treat countable homogeneous digraphs, including tournaments. Here, a \emph{digraph} is a graph where $a\to b$, $b\to a$, or neither, \emph{but not both}.  We leave three technical cases of countable homogeneous digraphs for a future note.

\section{Linear and local orders}
\label{sec:LinearAndLocalOrders}

\subsection{Linear orders}
\label{sec:linear}

There is only one countable homogeneous linear order, called the rational order $\QQ$. This is perhaps the best-known nontrivial homogenous structure because it is the unique countable dense linear order without endpoints.  Foreman has shown in \cite[Theorem~76]{F} that the conjugacy relation on $\Aut(\QQ)$ is Borel complete. We present here a slightly streamlined variant of his proof, since the details will be useful in the next subsection.

\begin{thm}[\protect{\cite[Theorem~76]{F}}]\label{thm:foreman}
  The isomorphism relation on countable linear orders is Borel reducible to the conjugacy relation on $\Aut(\QQ)$. Hence the conjugacy relation on $\Aut(\QQ)$ is Borel complete.
\end{thm}

\begin{proof}
  We must construct a Borel map $L\mapsto\phi_L$ from the set of linear orders on $\NN$ into $\Aut(\QQ)$ which satisfies:
\[\text{$L$ is isomorphic to $L'$}\iff\text{$\phi_L$ is conjugate to $\phi_{L'}$}\;.
\]
To ensure that $(\Leftarrow)$ holds, i.e., that $L$ can be recovered up to isomorphism from the conjugacy class of $\phi_L$, we simply arrange that the fixed point set of $\phi_L$ is isomorphic to $L$.  The main point in guaranteeing $(\Rightarrow)$ is to make sure that if $L$ and $L'$ are isomorphic, then the linear orderings of orbitals of $\phi_L$ and $\phi_{L'}$ will be isomorphic.

Here, the \emph{orbitals} of $\phi\in\Aut(\QQ)$ are the convex closures of the orbits $\set{\phi^n(q):n\in\ZZ}$.  Evidently, every orbital $R$ of $\phi$ is either:
\begin{itemize}
\item an ``up-bump:'' for all $q\in R$ we have $\phi(q)>q$;
\item a ``down-bump:'' for all $q\in R$ we have $\phi(q)<q$; or
\item a singleton which is a fixed point of $\phi$.
\end{itemize}
\noindent What we need is the following classical result:

\begin{lem}[\protect{\cite[Theorem~2.2.5]{G}}]\label{lem:glass}
  Let $\phi,\psi\in\Aut(\QQ)$ and suppose that there is an order-preserving bijection between the orbitals of $\phi$ and the orbitals of $\psi$ which is also \emph{type preserving}, in the sense that it sends up-bumps to up-bumps, down-bumps to down-bumps, and fixed points to fixed points.  Then $\phi$ and $\psi$ are conjugate in $\Aut(\QQ)$.
\end{lem}

Hence, to show $(\Rightarrow)$, it suffices to ensure that the order type (and type) of the orbitals of $\phi_L$ depends only on the order type of $L$.  For this, we will need to be a little bit careful:

\begin{lem}\label{lem:perfect}
  For any countable linear order $L$, there exists an order-preserving embedding $\alpha\colon L\to\QQ$ such that for every $q\in\QQ\smallsetminus\im(\alpha)$ there is a greatest element $q^-$ of $\im(\alpha)\cup\set{-\infty}$ below $q$ and a least element $q^+$ of $\im(\alpha)\cup\set{\infty}$ above $q$.
\end{lem}

\begin{proof}
  Let $\alpha_0\colon L\to\QQ$ be any embedding.  Letting $S$ be $\im(\alpha_0)$ together with the set of points $q\in\QQ\smallsetminus\im(\alpha_0)$ satisfying the desired property, it is easy to see that $S$ is a dense linear order without endpoints.  Hence there exists an isomorphism $i\colon S\to\QQ$, and now the composition $\alpha=i\circ\alpha_0$ is as desired.
\end{proof}

We now describe the construction of the Borel assignment $L\mapsto\phi_L$. Given the countable linear order $L$, let $\alpha_L\colon L\to\QQ$ be an embedding satisfying the property in Lemma~\ref{lem:perfect}.  We begin our definition of $\phi_L$ by declaring that it fixes every point of $\im(\alpha_L)$.  On the other hand, if $q\in\QQ\smallsetminus\im(\alpha_L)$, then we wish to define $\phi_L$ on the interval $(q^-,q^+)$ so as to guarantee that $(q^-,q^+)$ is an up-bump for $\phi_L$.  This can easily be done explicitly, for instance, using a piecewise linear function similar to the one pictured in Figure~\ref{fig:bump}.

\begin{figure}
  \begin{tikzpicture}
    \draw (0,0) -- (3.8,0);
    \draw (0,0) -- (0,3.2);
    \node[circle,draw,inner sep=1pt] at (1,1) (a) {};
    \node[circle,draw,inner sep=1pt] at (3,3) (b) {};
    \node[circle,fill,draw,inner sep=1pt] at (2,2.5) (c) {}
    edge (a) edge (b);
    \node[anchor=north] at (a) {$(a,a)$};
    \node[anchor=west] at (b) {$(b,b)$};
    \node[anchor=south east] at (c) {$(c,d)$};
  \end{tikzpicture}
  \caption{An ``up-bump'' on the interval $(a,b)$. Here one may take $c=.5a+.5b$ and $d=.25a+.75b$.\label{fig:bump}}
\end{figure}
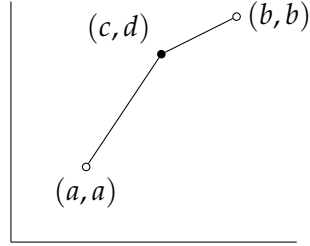

Since the fixed-point set of $\phi_L$ is exactly $\im(\alpha_L)$, we have guaranteed $(\Leftarrow)$. Since every remaining orbital of $\phi_L$ is an up-bump, the orbital structure of $\phi_L$ depends only on the order type of $L$.  Thus Lemma \ref{lem:glass} guarantees $(\Rightarrow)$.

Finally, we observe that our construction can be made explicit by fixing an enumeration of $\QQ$ in advance and using it to carry out all back-and-forth constructions.  In other words, we can ensure that the map $L\mapsto\phi_L$ is a Borel assignment. This completes the proof of Theorem~\ref{thm:foreman}.
\end{proof}

\subsection{Local orders and the structures $\bm{S(n)}$}
\label{sec:local}

The class of local orders is closely related to the class of linear orders. A \emph{local order} is a tournament with the property that for every $b$ both $\set{a\in T:a\rightarrow b}$ and $\set{c\in T:b\rightarrow c}$ are linearly ordered by $\to$. As was the case with linear orderings, there is a unique countable homogeneous local order called $\mathcal O$. See Section~6 of \cite{cameron} for more on local orders.

The structure $\mathcal O$ can be realized as one of a family of homogeneous structures $S(n)$, which are defined as follows. Begin with a fixed countable dense subset $D$ of the unit circle of the complex plane such that for every $x,y\in D$ neither $\arg(x)$ nor $\arg(x/y)$ is a rational multiple of $\pi$. For each fixed $n\geq2$, the structure $S(n)$ consists of $n$ binary relations $\to_k$ on $D$ defined by $x\to_ky$ iff $\arg(x/y)\in(2\pi k/n,2\pi(k+1)/n)$. (Of course only $\to_0,\ldots,\to_{\lceil n/2\rceil}$ are formally needed.) Each of the structures $S(n)$ is easily seen to be homogeneous.

Now the local order $\mathcal O$ can be defined from $S(2)$ by letting $x\to_{\mathcal O}y$ iff $x\to_0y$ for all $x,y\in D$. The structure $S(3)$ also gives rise to a homogeneous digraph on $D$ defined by $x\to y$ iff $x\to_0 y$. As we shall see in Section~\ref{sec:Digraphs}, the list of homogeneous digraphs does not include any structure corresponding to $S(n)$ for $n>3$.

\begin{thm}\label{thm:Sn}
  Let $n\geq2$. The isomorphism relation on countable linear orders is Borel reducible to the conjugacy relation on $\Aut(S(n))$.  Hence the conjugacy relation on $\Aut(S(n))$ is Borel complete.
\end{thm}





\begin{proof}
  Given a countable linear ordering $L$, we will define an automorphism $\phi_L$ of $S(n)$ in such a way that $L\cong L'$ iff $\phi_L$ and $\phi_{L'}$ are conjugate in $\Aut(S(n))$. Note that we lose no generality in assuming that $L$ has lower and upper endpoints.

  To begin, let $A_k=\set{x\mid 2\pi k/n<\arg(x)<2\pi(k+1)/n}$ denote the $k^\text{th}$ ``arc'' of the unit circle. Since $D\cap A_0$ is naturally linearly ordered by argument value (or $\to_0$), we may let $\alpha_L$ be an embedding from $L$ into $D\cap A_0$ which satisfies the property in Lemma~\ref{lem:perfect}.

  Next let $f$ be the map from the unit circle to $A_0$ defined by $f(x)=xe^{-2\pi ik/n}$ whenever $x\in A_k$. Notice that $f$ is one-to-one on the subset $D$, and also that $f(D)$ is naturally linearly ordered by argument value. It is also naturally colored by which sector the points came from, that is, for each $x\in A_k$ we assign $f(x)$ the color $k$.

  We now wish to define a color-preserving automorphism $\psi_L$ of the linear ordering $f(D)$ whose fixed point set is exactly the closure of $\im(\alpha_L)$, which has a down-bump below the minimum of $\im(\alpha_L)$, a down-bump above the maximum of $\im(\alpha_L)$, and up-bumps elsewhere. This can be done similarly to the previous proof, except that the bumps must be constructed by a back-and-forth argument to ensure they are color-preserving. Finally, we let $\phi_L=f^{-1}\circ\psi_L\circ f$ be the corresponding sector-preserving automorphism of $S(n)$.  Notice that $\phi_L \upharpoonright_{A_0}=\psi_L$.

  By Lemma~\ref{lem:glass}, if $L\cong L'$, then $\psi_L$ is conjugate to $\psi_{L'}$ and it follows that $\phi_L$ is conjugate to $\phi_{L'}$. On the other hand, any $\phi_{L}$ has just two special fixed points which are the endpoints of down-bumps, and we can recover $L$ as the linear order of fixed points that lie between (in circular order) these two special fixed points. Thus using the argument of Theorem~\ref{thm:foreman}, if $\phi_{L'}$ is conjugate to $\phi_L$ we must have $L'\cong L$.
\end{proof}

\section{Undirected graphs}
\label{sec:UndirectedGraphs}

Lachlan and Woodrow \cite{lachlan} classified the countably infinite homogeneous undirected graphs as follows:
\begin{itemize}
\item for $m,n\leq\infty$ and either $m$ or $n$ infinite, the graph $m\cdot K_n$ consisting of $m$ many disjoint copies of $K_n$ (section~\ref{sec:mKn});
\item the generic undirected graph, also known as the random graph (see \cite{summer});
\item for $n<\infty$, the generic $K_n$-free graph (section~\ref{sec:random}); and
\item graph complements of each of these (they have the same automorphism group).
\end{itemize}

\subsection{Composite undirected graphs}
\label{sec:mKn}

We first show that the classification of automorphisms of $m\cdot K_n$ is smooth when one of $m$ or $n$ is finite.

\begin{thm}\label{thm:Gn}
  If $m,n\leq\infty$ and either $m$ or $n$ is finite, then the conjugacy problem for the automorphism group of $m\cdot K_n$ is smooth.
\end{thm}

Before beginning the proof, observe that each automorphism $\phi\in\Aut(m\cdot K_n)$ acts on the set of copies of $K_n$ by an element $\phi'\in S_m$.  Recall that for $m\leq\infty$ the elements of $S_m$ are determined up to conjugacy by their \emph{cycle type}, that is, the sequence which tells the number of $k$-cycles for each $k\leq\infty$.  The situation is only slightly more complicated for elements of $\Aut(m\cdot K_n)$ since if $k<\infty$, each $k$-cycle of copies of $K_n$ contains an additional piece of information: the permutation of $K_n$ obtained by following the cycle from one copy of $K_n$ all the way around to the start.  More precisely, given $\phi\in\Aut(m\cdot K_n)$ and a $k$-cycle $Y_0,\ldots,Y_{k-1}$ of copies of $K_n$, we can consider $\phi^k\restriction Y_0$ as an element of $S_n=\Aut(Y_0)$. The \emph{twist type} of the cycle $Y_0,\ldots,Y_{k-1}$ is then the conjugacy equivalence class of $\phi^k\restriction Y_0$ in $S_n$.  This is well-defined since $\phi^j$ witnesses that $\phi^k\restriction Y_i$ and $\phi^k\restriction Y_{i+j}$ are conjugate.

\begin{proof}[Proof of Theorem~\ref{thm:Gn}]
  Let us first assume that $m=\infty$ and $n$ is finite. Let $T$ denote the set of conjugacy classes in $S_n=\Aut(K_n)$.  We claim that elements of $\Aut(\infty\cdot K_n)$ are classified up to conjugacy by the following invariants:
  \begin{itemize}
  \item for each $k<\infty$ and $t\in T$, the number of $k$-cycles of copies of $K_n$ with twist type equal to $t$; and
  \item the number of infinite cycles of copies of $K_n$.
  \end{itemize}
  It is easy to see that conjugate automorphisms will possess the same invariants.  Conversely, suppose that $\phi$ and $\psi$ have the same invariants.  Let $Y_0,\ldots,Y_{k-1}$ and $Z_0,\ldots Z_{k-1}$ be cycles of copies of $K_n$ for $\phi$ and $\psi$, respectively, and assume they have the same twist type.  Then there is a bijection $\delta_0\colon Y_0\to Z_0$ which satisfies $\delta_0\circ\phi^k=\psi^k\circ\delta_0$. This implies that it is well-defined to say: extend $\delta_0$ to a map $\delta$ on the entire cycle by letting $\delta(\phi^i(y))=\psi^i\circ\delta_0(y)$ for all $i<k$. Applying the same construction to each cycle, we can define $\delta$ on all of $\infty\cdot K_n$.  (For infinite cycles there is not even any twist type to worry about.)  It is easy to see that this $\delta$ is an automorphism of $\infty\cdot K_n$ and satisfies $\delta\circ\phi=\psi\circ\delta$.

  Next, we consider the case when $m$ is finite and $n=\infty$. In this case the set $T$ of conjugacy classes of $S_\infty=\Aut(K_\infty)$ is uncountable. But since $m$ is finite, each fixed automorphism only mentions a finite set of elements of $T$ as twist types of cycles of copies of $K_\infty$. Thus in this case the elements of $\Aut(m\cdot K_\infty)$ are classified by:
  \begin{itemize}
  \item the finite subset $T_0\subset T$ of elements realized as the twist type of some cycle of copies of $K_\infty$; and
  \item for each $k<\infty$ and $t\in T_0$, the number of $k$-cycles of copies of $K_\infty$ with twist type equal to $t$.
  \end{itemize}
  It is easy to show that a finite subset of $T$ can be coded by a single real number (for this, use a fixed linear ordering of $T$ to enumerate the finite set, and then use any Borel bijection $\bigcup_{i<\omega} T^i\to\RR$). Thus this is once again a smooth classification.
\end{proof}

Although the situation when $m=n=\infty$ is similar, in this case each automorphism may mention countably many elements from the uncountable set $T$ of twist types. This turns out to be at a higher level of complexity than the smooth relations, but still lower than the Borel complete relations. In this way the following result is unique among all the results in this paper.

Before stating the result, we let $E_\sset$ denote the equivalence relation on $\RR^\omega$ given by $\sigma\mathrel{E_\sset}\tau$ iff $\sigma$ and $\tau$ enumerate the same countable set. The Borel complexity of $E_\sset$ is known to lie properly in between the smooth and Borel complete complexities (see for example \cite{gao}, Chapter~8, where $E_\sset$ is denoted $=^+$). 

\begin{thm}\label{thm:Ginfty}
  The conjugacy problem for the automorphism group of the graph $\infty\cdot K_\infty$ is Borel bireducible with $E_\sset$.
\end{thm}

\begin{proof}
  Again let $T$ denote the set of conjugacy classes in $S_\infty=\Aut(K_\infty)$. The arguments of the previous proof imply that elements of $\Aut(\infty\cdot K_\infty)$ are classified by:
  \begin{itemize}
  \item the countable subset $T_0\subset T$ of elements realized as the twist type of some finite cycle of copies of $K_\infty$;
  \item for each $k<\infty$ and $t\in T_0$, the number of $k$-cycles of copies of $K_\infty$ with twist type equal to $t$; and
  \item the number of infinite cycles of copies of $K_\infty$.
  \end{itemize}
  We must verify that this implies the conjugacy problem for $\Aut(\infty\cdot K_\infty)$ is Borel bireducible with $E_\sset$. To see that the conjugacy problem is Borel reducible to $E_\sset$, note that we can code the invariant above using a countable subset of $T\times(\NN\cup\{\infty\})^3$. Indeed, given $\phi$, form the set of all $(t,k,l,i)$ where $t$ is a twist type occurring in $\phi$, $k\in\NN$, $l$ is the number of $k$-cycles of copies of $K_\infty$ with twist type equal to $t$, and $i$ is the number of infinite cycles of copies of $K_\infty$.

  We next reduce $E_\sset$ to the conjugacy problem for $\Aut(\infty\cdot K_\infty)$ as follows. Given a countable subset $T_0\subset T$, we form an automorphism $\phi$ of $\infty\cdot K_\infty$ which has $|T_0|$ many $2$-cycles of copies of $K_\infty$, no other cycles, and such that each $t\in T_0$ appears exactly once as a twist type.
\end{proof}

\subsection{Random graphs}
\label{sec:random}

In this section we discuss the generic graph, known as the \emph{random graph} $\Gamma$, as well as the generic  $K_n$-free graph denoted $\Gamma_n$. Here if $\mathcal C$ is a class of finite graphs (or digraphs, or relational structures) we say $G$ is \emph{generic} for the class $\mathcal C$ if $G$ is homogeneous and the set of finite substructures of $G$ is exactly $\mathcal C$. The classes $\mathcal C$ which admit a generic object are characterized by the well-known Fra\"iss\'e theory.

When dealing with generic objects, we will often use the following characterization, known as the \emph{one-point extension property}. This states that $G$ is generic for the class $\mathcal C$ if and only if every finite subset $S\subset G$ lies in $\mathcal C$, and whenever $S\cup\{x\}$ lies in $\mathcal C$ there is some $a\in G$ such that the identity function on $S$ extends to an isomorphism $S\cup\{x\}\cong S\cup\{a\}$.

In the article \cite{summer} we showed with Scott Schneider that the conjugacy problem for $\Aut(\Gamma)$ is Borel complete. The next result gives a streamlined version of the argument from \cite{summer}, and at the same time generalizes it to work for $\Aut(\Gamma_n)$ too.

\begin{thm}\label{thm:graphs}
  Let $n\geq3$. The isomorphism relation for countable $K_n$-free graphs is Borel reducible to the conjugacy problem for $\Aut(\Gamma_n)$. Hence the conjugacy problem for $\Aut(\Gamma_n)$ is Borel complete.
\end{thm}

\begin{proof}
  Given a countable $K_n$-free graph $G$, we construct a copy $\Delta_G$ of $\Gamma_n$ together with an automorphism $\phi_G$ of $\Delta_G$. It is enough to show that $G\cong G'$ iff $\phi_G$ and $\phi_{G'}$ are conjugate by an isomorphism $\Delta_G\cong\Delta_{G'}$.

  To begin, let $\Delta_G^0$ consist of two disjoint copies of $G$, with each vertex adjacent to the corresponding vertex in the other copy. Also, let $\phi_G^0$ be the automorphism of $\Delta_G^0$ which exchanges corresponding vertices from the two copies of $G$.

  Next suppose $\Delta_G^k$ and $\phi_G^k$ have been constructed and define $\Delta_G^{k+1}\supset\Delta_G^k$ as follows. For each finite subset $S\subset \Delta_G^k$ which does not contain a copy of $K_{n-1}$, we place a point $x$ into $\Delta_G^{k+1}$ which is adjacent to every vertex of $S$ and no other vertices in $\Delta_G^{k+1}$. Then let $\phi_G^{k+1}$ be the unique extension of $\phi_G^k$ to an automorphism of $\Delta_G^{k+1}$.

  To complete the construction, we let $\Delta_G=\bigcup\Delta_G^k$ and $\phi_G=\bigcup\phi_G^k$. It is clear that $\Delta_G$ has the one-point extension property relative to the class of $K_n$-free graphs and hence that it is a copy of $\Gamma_n$. Moreover, if $G\cong G'$ then this extends to an isomorphism $\Delta_G^0\cong\Delta_{G'}^0$, and this uniquely extends layer-by-layer to an isomorphism $\alpha\colon\Delta_G\cong\Delta_{G'}$. It is easy to verify that this isomorphism satisfies $\alpha\phi_G=\phi_{G'}\alpha$.

  For the converse, first note from the construction that if $x$ lies in $\Delta_G^0$ then $x$ is adjacent to $\phi_G(x)$, while if $x$ lies in some $\Delta_G^{k+1}\smallsetminus\Delta_G^k$ then so does $\phi_G(x)$ and hence $x$ is \emph{not} adjacent to $\phi_G(x)$. Thus if we are given $\phi_G$ we can recover $\Delta_G^0$ as the set of vertices $x$ such that $x$ is adjacent to $\phi_G(x)$. And we can further recover $G$ as the quotient graph of $\Delta_G^0$ by the orbit equivalence relation on $\phi_G$.

  Now if $\alpha\colon\Delta_G\cong\Delta_{G'}$ and $\alpha\phi_G=\phi_{G'}\alpha$ it follows that $\alpha$ restricts to an isomorphism $\Delta_G^0\cong\Delta_{G'}^0$ that sends $\phi_G$-orbits to $\phi_{G'}$-orbits. Therefore by passing to the quotient graphs of $\Delta_G^0,\Delta_{G'}^0$ by the $\phi_G$ and $\phi_{G'}$-orbit equivalence relations, we see that $\alpha$ induces an isomorphism  $G\cong G'$.

  To conclude, we remark briefly on how the construction can be exhibited in a Borel fashion. We fix the underlying sets of $G,\Delta_G,\Gamma_n$ to be $\NN$. The construction of $\Delta_G$ can be made Borel by reserving an infinite subset $I_k\subset\NN$ for each $\Delta_G^k$, and using a previously fixed enumeration of the finite subsets $S\subset I_k$. This immediately implies that the construction of $\phi_G$ is Borel also. Finally we can regard $\phi_G$ as an automorphism of $\Gamma_n$ using a back-and-forth construction between $\Delta_G$ and $\Gamma_n$, where each choice in the construction is resolved by choosing the least available witness.
\end{proof}

\section{Digraphs}
\label{sec:Digraphs}

For us, a \emph{digraph} is an antisymmetric and irreflexive binary relation.
The countable homogeneous digraphs have been classified by Cherlin \cite{cherlin}. The following catalog of these digraphs also serves as a table of contents for this section.

\begin{itemize}
\item We have already mentioned $\QQ$, $S(2)$, and $S(3)$, which can all be viewed as digraphs (sections~\ref{sec:linear} and~\ref{sec:local})
\item The generic tournament $\mathcal T$ (section~\ref{sec:tournament})
\item Generic independent set avoiding digraphs $\Lambda_n$ (section~\ref{sec:tournament})
\item Compositions of certain tournaments with $I_n$ (section~\ref{sec:composite})
\item Slight modifications of certain tournaments $\hat T$ (section~\ref{sec:hat})
\item Generic tournament-avoiding digraphs $\Gamma_{\mathcal F}$ (section~\ref{sec:Tfree})
\item Generic complete multipartite digraphs (section~\ref{sec:multipartite})
\item Semigeneric multipartite digraph \cite{cherlin-imprimitive} (not treated) 
\item Generic partial order $\mathcal P$ (not treated)
\item Shuffled generic partial order $P(3)$ (not treated)
\end{itemize}

There are also several finite examples, but the conjugacy problems for their automorphism groups are all clearly smooth.

\subsection{The random tournament and universal $I_n$-free digraphs.}
\label{sec:tournament}

There is a generic countable tournament $\mathcal T$, sometimes also called the \emph{random tournament}.

\begin{thm}\label{thm:tournament}
  The isomorphism relation for countable linear orders is Borel reducible to the conjugacy problem for $\Aut(\mathcal T)$. Hence the conjugacy problem for $\Aut(\mathcal T)$ is Borel complete.
\end{thm}

\begin{proof}
  We employ a similar method to the proof of Theorem~\ref{thm:graphs}, adapting some of the combinatorial details to this situation. Beginning with a linear order $L$ we again construct a copy $\Delta_L$ of $\mathcal T$ together with an automorphism $\phi_L$ of $\Delta_L$ in such a way that $L\cong L'$ iff $\phi_L$ is conjugate to $\phi_{L'}$. As before, the construction can easily be arranged to be Borel.

  To begin, we let $\Delta_L^0$ consist of \emph{three} copies of $L$, where for each vertex $x\in L$ we place three vertices $x_0,x_1,x_2$ into $\Delta_L^0$ with $x_0\to x_1\to x_2\to x_0$. For each edge $x\to y$ of $L$ we place the nine edges $x_i\to y_j$ into $\Delta_L^0$. We then let $\phi_L^0$ be the automorphism of $\Delta_L^0$ that maps the vertices of $\Delta_L^0$ in the fashion $x_0\mapsto x_1\mapsto x_2\mapsto x_0$ so that in all cases $x_i\to \phi_L^0(x_i)$. Finally we extend the linear ordering of $L$ to an ordering $<_L^0$ of $\Delta_L^0$ by letting $x_0<x_1<x_2<x_3$ and $x_i<y_j$ whenever $x<y$ in $L$.

  Now suppose that $\Delta_L^k$, $\phi_L^k$ , and $<_L^k$ have been constructed and define $\Delta_L^{k+1}$ as follows. For each finite subset $S\subset\Delta_L^k$ we place a vertex $x$ into $\Delta_L^{k+1}$ such that $s\to x$ for all $s\in S$ and $a\leftarrow x$ for all $a\in\Delta_L^k\smallsetminus S$. Then there is a unique automorphism $\phi_L^{k+1}$ of $\Delta_L^{k+1}$ which extends $\phi_L^k$. We also extend the linear order $<_L^k$ to $<_L^{k+1}$ as follows: if $x,x'$ are the vertices corresponding to the finite sets $S,S'$, then we set $x<_L^{k+1}x'$ iff $S<S'$ in the lexicographic order on finite sets derived from $<_L^k$.

  We still need to add edges within $\Delta_L^{k+1}\smallsetminus\Delta_L^k$ to make $\Delta_L^{k+1}$ a tournament. First, within each nontrivial $\phi_L^{k+1}$-orbit of $\Delta_L^{k+1}\smallsetminus\Delta_L^k$ we make a copy of $C_3$ by adding the edges $x\leftarrow\phi_L^{k+1}(x)$.  Second, if $\{x_i\}$ and $\{y_j\}$ are distinct $\phi_L^{k+1}$-orbits within $\Delta_L^{k+1}\smallsetminus\Delta_L^k$, we either add all the edges $x_i\to y_j$ or all the edges $x_i\leftarrow y_j$. This choice can be made systematic: if $\min\{x_i\}<\min\{y_j\}$ in the $<_L^{k+1}$ ordering, then we set $x_i\to y_j$.


  It is easy to see that if $L\cong L'$ then $\phi_L$ is conjugate to $\phi_{L'}$. Indeed, if $\alpha$ is an isomorphism $L\cong L'$, then by induction $\alpha$ induces an isomorphism $\Delta_L^k\cong\Delta_{L'}^k$ for each $k$, and this induced isomorphism conjugates $\phi_L$ to $\phi_{L'}$. Moreover, given $\phi_L$, we can recover $\Delta_L^0$ as the set of vertices $x$ such that $x\to\phi_L(x)$. It follows that we can conclude exactly as in the proof of Theorem~\ref{thm:graphs}.
\end{proof}

Just as the random graph $\Gamma$ admitted a family of $K_n$-free generalizations $\Gamma_n$, the random tournament $\mathcal T$ admits a family of $I_n$-free generalizations $\Lambda_n$. (Here, recall that $I_n$ denotes an edgeless digraph with $n$ vertices.) With this notation, $\Lambda_2$ is just $\mathcal{T}$ itself.

\begin{thm}\label{thm:I_n-free digraphs}
  Let $n\geq2$. The isomorphism relation for countable linear orders is Borel reducible to the conjugacy problem for $\Aut(\Lambda_n)$. Hence the conjugacy problem for $\Aut(\Lambda_n)$ is Borel complete.
\end{thm}

\begin{proof}
  We explain how to modify the previous proof to work for this family of digraphs. Once again suppose that $L$ is a linear order and that $\Delta_L^k$, $\phi_L^k$, and $<_L^k$ have been constructed. This time, for each pair of disjoint finite subsets $S,S'\subset\Delta_L^k$ such that $S'$ does not contain an independent set of size $n-1$, we add a vertex $x$ to $\Delta_L^{k+1}$ such that $s\to x$ for all $s\in S$, $s$ is not adjacent to $x$ for all $s\in S'$, and $a\leftarrow x$ for all $a\in\Delta_L^k\smallsetminus(S\cup S')$. In this way we realize all types over $\Delta_L^k$ that do not violate the $I_n$-free property. The rest of the construction proceeds as in the previous proof, except of course we define $<_L^{k+1}$ using the lexicographic ordering on pairs $(S,S')$.

  The remainder of the argument is the same as before. We can argue similarly that $\Delta_L$ is a copy of $\Lambda_n$, $\phi_L$ is an automorphism of $\Delta_L$, and the map $L\mapsto\phi_L$ gives a Borel reduction from isomorphism of linear orders to conjugacy in $\Aut(\Delta_L)=\Aut(\Lambda_n)$.
\end{proof}

\subsection{Composite digraphs}
\label{sec:composite}

For any digraph $G$ and $n\leq\infty$, we let $n\cdot G$ denote the digraph with $n$ disjoint copies of $G$. We also let $G[n]$ denote $G$ with each vertex replaced by an independent set of size $n$, where the edges between the independent sets are determined by the edges of $G$. Then there are eight classes of homogeneous composite digraphs:

\begin{itemize}
\item $\infty\cdot C_3$,\ \ \ $C_3[\infty]$
\item $n\cdot\QQ$,\ \ \ $\QQ[n]$
\item $n\cdot S(2)$,\ \ \ $S(2)[n]$
\item $n\cdot\mathcal T$,\ \ \ $\mathcal T[n]$
\end{itemize}

The following result settles the complexity of the conjugacy problem for the automorphism groups of each of these digraphs.

\begin{thm}
  \begin{itemize}
  \item The conjugacy problems for $\Aut(\infty\cdot C_3)$ and $\Aut(C_3[\infty])$ are both smooth.
  \item The conjugacy problems for the remaining digraphs in the list above are all Borel complete. Indeed, if $G$ is a tournament and the conjugacy problem for $\Aut(G)$ is Borel complete, then the conjugacy problems for $\Aut(n\cdot G)$ and $\Aut(G[n])$ are Borel complete.
  \end{itemize}
\end{thm}

\begin{proof}
  To show that $\Aut(\infty\cdot C_3)$ is smooth, we can use an argument identical to the one in Theorem~\ref{thm:Gn}. Here, the ``twist types'' are simply the three elements of $\Aut(C_3)$. The argument for $\Aut(C_3[\infty])$ is similar, since any element of $\Aut(C_3[\infty])$ acts on the copies of $I_\infty$ by an automorphism of $C_3$. And as with the previous argument, each cycle of copies of $I_\infty$ has an associated ``twist type'' which is a conjugacy class of $S_\infty=\Aut(I_\infty)$.

  Next, if conjugacy in $\Aut(G)$ is Borel complete, let $\phi\oplus\mathrm{id}$ denote the automorphism of $n\cdot G$ which acts by $\phi$ on the first copy of $G$ and trivially on the remaining copies. Then it is easy to see that since $G$ is connected, the map $\phi\mapsto\phi\oplus\mathrm{id}$ is a reduction from conjugacy in $\Aut(G)$ to conjugacy in $\Aut(n\cdot G)$.

  Finally, we let $\phi[n]$ denote the automorphism of $G[n]$ which acts by $\phi$ on the copies of $I_n$ and acts trivially within copies of $I_n$. Once again, it is easy to check that since $G$ is a tournament the map $\phi\mapsto\phi[n]$ is a reduction from conjugacy in $\Aut(G)$ to conjugacy in $\Aut(G[n])$.
\end{proof}

We conjuncture that the above result may be strengthened, either by weakening the hypotheses on the digraph $G$ or by generalizing it to a larger class of countable structures.

\subsection{Hat graphs}
\label{sec:hat}

Given a tournament $T$, we define $\hat T$ as follows: let $a$ be a new point and let $\hat T$ initially consist of two disjoint copies of $a\to T$, call them $a\to T$ and $\bar{a}\to\bar{T}$. Given points $x\in T\cup\{a\}$ and $y\in\bar{T}\cup\{\bar{a}\}$, we let $x\to\bar{y}$ if $x\leftarrow y$ and $x\leftarrow\bar{y}$ if $x\to y$.

The automorphism group of $\hat T$ is generated by $\Aut(T)$ together with a rather trivial automorphism swapping the two copies.  If $T$ is infinite, then the digraph $\hat T$ is homogeneous in only two cases: $T=\QQ$ and $T=\mathcal{T}$.  In each of these cases, the conjugacy relation is Borel complete, and it follows that the conjugacy relation in $\hat T$ is also Borel complete.

\subsection{Generic tournament-avoiding digraphs}
\label{sec:Tfree}

While the random graph $\Gamma$ had generic $K_n$-free variants $\Gamma_n$, the generic countable digraph has a family of continuum many variants. For any family $\mathcal{F}$ of finite tournaments (each of size $\geq3$), we say that a digraph $G$ is \emph{$\mathcal{F}$-free} if it does not contain a copy of any element of $\mathcal F$. For each such family $\mathcal{F}$ there is a universal countable homogeneous such digraph $\Gamma_{\mathcal{F}}$. In the case that $\mathcal{F}=\emptyset$, the resulting digraph $\Gamma_{\mathcal{F}}$ is called the \emph{random digraph}.

\begin{thm}\label{thm:digraphs}
  If $\mathcal{F}$ is a family of finite tournaments, each of size $\geq3$, then the isomorphism problem for the class of $\mathcal F$-free digraphs is Borel reducible to the conjugacy problem for $\Aut(\Gamma_{\mathcal{F}})$.  Hence the conjugacy problem for $\Aut(\Gamma_{\mathcal{F}})$ is Borel complete.
\end{thm}

\begin{proof}
  We combine the arguments in the proofs of Theorems~\ref{thm:graphs} and \ref{thm:tournament}. Given a countable $\mathcal{F}$-avoiding digraph $G$ we construct a copy $\Delta_G$ of $\Gamma_{\mathcal F}$ and an automorphism $\phi_G$ of $\Delta_G$. This time we let $\Delta_G^0$ consist of \emph{four} copies of $G$, where for each $x\in G$ we place vertices $x_0,\ldots,x_3$ into $\Delta_G^0$ with $x_0\to x_1\to x_2\to x_3\to x_0$. For each edge $x\to y$ in $G$ we place the four edges $x_i\to y_i$ into $\Delta_G^0$. Note that the only tournaments in $\Delta_G^0$ are those already present in $G$. We then let $\phi_G^0$ be the automorphism of $\Delta_G^0$ that maps $x_0\mapsto\cdots\mapsto x_3\mapsto x_0$ so that in all cases $x_i\to\phi_G^0(x_i)$.

  Now suppose that $\Delta_G^k$ and $\phi_G^k$ have been constructed and define $\Delta_G^{k+1}$ and $\phi_G^{k+1}$ as follows. For each finite subset $S\subset\Delta_G^k$ we provisionally place a vertex $x$ into $\Delta_G^{k+1}$ such that $s\to x$ for all $s\in S$ and $a\leftarrow x$ for all $a\in\Delta_G^k\smallsetminus S$. However, if doing so would create a copy of some $T\in\mathcal{F}$, we simply skip adding the element $x$ instead. As in the proof of Theorem~\ref{thm:tournament}, we let $\phi_G^{k+1}\supset\phi_G^k$ be the unique extension to an automorphism of $\Delta_G^{k+1}$, and add edges within the $\phi_G^{k+1}$-orbits of $\Delta_G^{k+1}\smallsetminus\Delta_G^k$ in such a way that each nontrivial orbit is a copy of $C_4$ where $x\leftarrow\phi_G^{k+1}(x)$. We don't add edges between the orbits. Otherwise the conclusion of the proof is now the same as in the proof of Theorem~\ref{thm:tournament}.
\end{proof}

\subsection{Generic complete $n$-partite digraphs}
\label{sec:multipartite}

A digraph is said to be \emph{complete $n$-partite} if it is $n$-partite and has a maximal set of edges.  For each $2\leq n\leq\infty$, there exists a generic such digraph, which we denote $n*I_\infty$.

\begin{thm}
  The isomorphism relation for countable linear orders is Borel reducible to the conjugacy problem for $\Aut(n*I_\infty)$. Hence the conjugacy problem for $\Aut(n*I_\infty)$ is Borel complete.
\end{thm}

\begin{proof}
  We begin by treating the special case when $n=2$.  Given a linear order $L$, we build a copy $\Delta_L$ of $2*I_\infty$ and an automorphism $\phi_L$ of $\Delta_L$. We let $\Delta_L^0$ consist of four copies of $L$, where for each element $x\in L$ we place four vertices $x_0,\ldots,x_3$ into $\Delta_L^0$ with edges $x_0\to\cdots\to x_3\to x_0$. For each pair $x<y$ of $L$ we place \emph{eight} edges $x_{2i}\to y_{2j+1}$ and $x_{2i+1}\to y_{2j}$ into $\Delta_L^0$ (this is depicted in Figure~\ref{fig:bipartite}). Since $L$ is in particular a tournament, we have that $\Delta_L^0$ is a complete bipartite digraph. Next we let $\phi_L^0$ be the automorphism of $\Delta_L^0$ that maps the vertices of $\Delta_L^0$ in the fashion $x_0\mapsto\cdots\mapsto x_3\mapsto x_0$, so that in all cases we have $x_i\to\phi_L^0(x_i)$. We also extend the linear ordering of $L$ to an ordering $<_L^0$ of $\Delta_L^0$ by letting $x_0<x_1<x_2<x_3$ and $x_i<y_j$ whenever $x<y$ in $L$.

  \begin{figure}[h]
    \centering
    \begin{tikzpicture}[->]
      \node at (-4,0) {$x<y$};
      \node at (-2.5,0) {becomes};
      \node (x0) at (-1,1) {$x_0$};
      \node (x1) at (1,1) {$x_1$};
      \node (x2) at (-1,.5) {$x_2$};
      \node (x3) at (1,.5) {$x_3$};
      \node (y0) at (-1,-.5) {$y_0$};
      \node (y1) at (1,-.5) {$y_1$};
      \node (y2) at (-1,-1) {$y_2$};
      \node (y3) at (1,-1) {$y_3$};
      \draw (x0)--(x1);\draw (x1)--(x2);\draw (x2)--(x3);\draw (x3)--(x0);
      \draw (y0)--(y1);\draw (y1)--(y2);\draw (y2)--(y3);\draw (y3)--(y0);
      \draw[-,shorten >=2pt,double,double distance=2.4pt] (-.4,.4) -- (.4,-.4);
      \draw[double] (-.4,.4) -- (.4,-.4);
      \draw[-,shorten >=2pt,double,double distance=2.4pt] (.4,.4) -- (-.4,-.4);
      \draw[double] (.4,.4) -- (-.4,-.4);
    \end{tikzpicture}
    \caption{The construction of $\Delta_L^0$ from $L$.\label{fig:bipartite}}
  \end{figure}
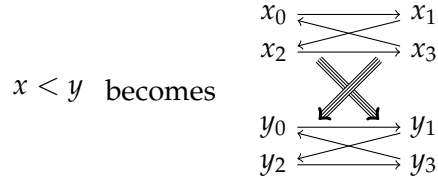

  Now suppose that $\Delta_L^k$ and $\phi_L^k$ have been constructed and inductively suppose that $\Delta_L^k$ consists of two maximal independent sets $A_0$ and $A_1$. We build extensions $\Delta_L^{k+1}\supset\Delta_L^k$ and $\phi_L^{k+1}\supset\phi_L^k$ as follows. For each $A_i$ and each finite subset $S\subset A_{1-i}$, we put a new point $x$ into $\Delta_L^{k+1}$ such that $s\to x$ for all $s\in S$ and $a\leftarrow x$ for all $a\in A_{1-i}\smallsetminus S$. As usual we let $\phi_L^{k+1}$ be the unique extension of $\phi_L^k$ to $\Delta_L^{k+1}$. We then add edges within each $\phi_L^{k+1}$-orbit in $\Delta_L^{k+1}\smallsetminus\Delta_L^k$ so as to ensure $x\leftarrow\phi_L^{k+1}(x)$ always holds.

  Finally, we fill in the remaining edges between the $\phi_L^{k+1}$-orbits in $\Delta_L^{k+1}\smallsetminus\Delta_L^k$ similarly to the proof of Theorem~\ref{thm:tournament}. More specifically, we again define a linear ordering $<_L^{k+1}$ from $<_L^k$ using the lexicographic ordering of finite sets $S$. Then if $\{x_i\}$ and $\{y_i\}$ are distinct $\phi_L^{k+1}$-orbits then we add all the edges from $\{x_i\}\cap A_0$ to $\{y_i\}\cap A_1$ and from $\{x_i\}\cap A_1$ to $\{y_i\}\cap A_0$ precisely when $\min\{x_i\}<\min\{y_i\}$. The conclusion of the proof when $n=2$ is now just the same as in the proof of Theorem~\ref{thm:tournament}.

  We now briefly say how to modify the above argument in the case when $n>2$. This time we inductively suppose that $\Delta_L^k$ and $\phi_L^k$ have been constructed and $\Delta_L^k$ consists of $n$ maximal independent sets $A_i$ for $0\leq i<n$. (In the step $k=0$, the $A_i$ will be empty for $2\leq i<n$.) We define $\Delta_L^{k+1}$ as follows: for each $i$ and each finite subset $S\subset\Delta_L^{k}$ that does not meet $A_i$ we add a new point $x$ to $A_i$ such that $s\to x$ for all $s\in S$ and $a\leftarrow x$ for all $a\in\Delta_L^k\smallsetminus(S\cup A_i)$. We then make $\Delta_L^{k+1}\smallsetminus\Delta_L^k$ complete $n$-partite by proceeding as in the case when $n=2$ within $A_0\cup A_1$, and additionally adding edges from $A_i$ to $A_j$ for $i<j$ when $2\leq j$. This guarantees that there is a unique extension to $\phi_L^{k+1}\supset \phi_L^k$ to $\Delta_L^{k+1}$ that interchanges $A_0$ and $A_1$ and preserves $A_i$ for $i\geq2$. The rest of the proof is the same as above.
\end{proof}

\bibliographystyle{alpha}
\begin{singlespace}
  \bibliography{summer,conjugacy}

\begin{thebibliography}{KPT05}

\bibitem[Cam81]{cameron}
Peter~J. Cameron.
\newblock Orbits of permutation groups on unordered sets. {II}.
\newblock {\em J. London Math. Soc. (2)}, 23(2):249--264, 1981.

\bibitem[CES11]{summer}
Samuel Coskey, Paul Ellis, and Scott Schneider.
\newblock The conjugacy problem for the automorphism group of the random graph.
\newblock {\em Arch. Math. Logic}, 50(1-2):215--221, 2011.

\bibitem[Che87]{cherlin-imprimitive}
Gregory~L. Cherlin.
\newblock Homogeneous directed graphs. {T}he imprimitive case.
\newblock In {\em Logic colloquium '85 ({O}rsay, 1985)}, volume 122 of {\em
  Stud. Logic Found. Math.}, pages 67--88. North-Holland, Amsterdam, 1987.

\bibitem[Che98]{cherlin}
Gregory~L. Cherlin.
\newblock The classification of countable homogeneous directed graphs and
  countable homogeneous {$n$}-tournaments.
\newblock {\em Mem. Amer. Math. Soc.}, 131(621):xiv+161, 1998.

\bibitem[For00]{F}
Matthew Foreman.
\newblock A descriptive view of ergodic theory.
\newblock In {\em Descriptive set theory and dynamical systems
  (Marseille-Luminy, 1996)}, volume 277 of {\em London Math. Soc. Lecture Note
  Ser.}, pages 87--171. Cambridge Univ. Press, Cambridge, 2000.

\bibitem[FS89]{FS}
Harvey Friedman and Lee Stanley.
\newblock A {B}orel reducibility theory for classes of countable structures.
\newblock {\em J. Symbolic Logic}, 54(3):894--914, 1989.

\bibitem[Gao09]{gao}
Su~Gao.
\newblock {\em Invariant descriptive set theory}, volume 293 of {\em Pure and
  Applied Mathematics (Boca Raton)}.
\newblock CRC Press, Boca Raton, FL, 2009.

\bibitem[Gla81]{G}
A.~M.~W. Glass.
\newblock {\em Ordered permutation groups}, volume~55 of {\em London Math. Soc.
  Lecture Note Ser.}
\newblock Cambridge University Press, Cambridge, 1981.

\bibitem[HK96]{HK}
Greg Hjorth and Alexander~S. Kechris.
\newblock Borel equivalence relations and classifications of countable models.
\newblock {\em Ann. Pure Appl. Logic}, 82(3):221--272, 1996.

\bibitem[KPT05]{kpt}
A.~S. Kechris, V.~G. Pestov, and S.~Todorcevic.
\newblock Fra\"\i ss\'e limits, {R}amsey theory, and topological dynamics of
  automorphism groups.
\newblock {\em Geom. Funct. Anal.}, 15(1):106--189, 2005.

\bibitem[Las93]{lascar}
D.~Lascar.
\newblock The group of automorphisms of a relational saturated structure.
\newblock In {\em Finite and infinite combinatorics in sets and logic ({B}anff,
  {AB}, 1991)}, volume 411 of {\em NATO Adv. Sci. Inst. Ser. C Math. Phys.
  Sci.}, pages 225--236. Kluwer Acad. Publ., Dordrecht, 1993.

\bibitem[LW80]{lachlan}
A.~H. Lachlan and Robert~E. Woodrow.
\newblock Countable ultrahomogeneous undirected graphs.
\newblock {\em Trans. Amer. Math. Soc.}, 262(1):51--94, 1980.

\bibitem[Mac11]{macpherson}
Dugald Macpherson.
\newblock A survey of homogeneous structures.
\newblock {\em Discrete Math.}, 311(15):1599--1634, 2011.

\end{thebibliography}
\end{singlespace}

\end{document}